\documentclass[footinclude=true]{amsart}
\usepackage[margin=1.5in]{geometry}
\usepackage{amsmath}
\usepackage{amsthm}
\usepackage{graphicx}
\usepackage{amssymb}
\usepackage{epstopdf}
\usepackage{nicefrac}
\usepackage{mathrsfs}
\usepackage{mathrsfs}
\usepackage{hyperref}
\usepackage{enumitem} 
\theoremstyle{plain}

\newtheorem{theo}{Theorem}
\newtheorem{lm}[theo]{Lemma}
\newtheorem*{q}{Question}
\newtheorem*{Claim}{Claim}

\newtheorem*{defin}{Definition}

\newtheorem{cor}[theo]{Corollary}
\newtheorem{prop}[theo]{Proposition}

\newtheorem*{example}{Example}
\newtheorem*{lm*}{Lemma}
\DeclareMathOperator{\Ha}{H} 
\DeclareMathOperator{\cone}{cone} 
\DeclareMathOperator{\conv}{conv} 
\DeclareMathOperator{\PW}{PW}

\DeclareMathOperator{\supp}{supp} 
\DeclareMathOperator{\ext}{ext} 
 
\title{On the failure of the Nehari Theorem for Paley-Wiener spaces}
\author{Konstantinos Bampouras}
\address{Department of Mathematical Sciences, Norwegian University of Sciences and Technology (NTNU), NO-7491 Trondheim, Norway}
\email{konstantinos.bampouras@ntnu.no}


\date{\today}

\usepackage{cite}
\begin{document}
	\maketitle
	\begin{abstract}
	  Let $\Omega$ be a nonempty, open and convex subset of $\mathbb{R}^{n}$. The Paley-Wiener space with respect to $\Omega$ is defined to be the closed subspace of $L^{2}(\mathbb{R}^{n})$ of functions with Fourier transform supported in $\Omega$. For a tempered distribution $\phi$, we define a Hankel operator to be the densely defined operator given by
	 $$\widehat{\Ha_{\phi}f}(x)=\int_{\Omega}\widehat{f}(y)\widehat{\phi}(x+y)dy,\text{   for $x\in\Omega$}.$$ We say that the Nehari theorem is true for $\Omega$, if every bounded Hankel operator is generated by a bounded function. In this paper we prove that the Nehari theorem fails for any convex set in $\mathbb{R}^{n}$ that has infinitely many extreme points. In particular, it fails for all convex bounded sets which are not polytopes. Furthermore, in the setting of $\mathbb{R}^{2}$, it fails for all non-polyhedral sets, bounded or unbounded.
	\end{abstract}

\section{Introduction}
\par
 Let $\Omega$ be a non-empty convex and open subset of $\mathbb{R}^{n}$. Given a function $f\in L^{1}(\mathbb{R}^{n})$ we will denote by $\widehat{f}$ or $\mathscr{F}(f)$ the Fourier transform of $f$, i.e. the function
 $$ \widehat{f}(x)=\int_{\mathbb{R}^{n}}f(y)e^{-2\pi ixy}dy,\quad\text{$x\in\mathbb{R}^{n}$}.$$ It is known that the Fourier transform can be extended to an isometric isomorphism of $L^{2}(\mathbb{R}^{n})$ onto itself that maps the Schwartz space onto itself and has inverse $\mathcal{F}^{-1}(f)(x)=\widehat{f}(-x)$. The Paley-Wiener space with respect to some open set $\Omega\subset\mathbb{R}^{n}$ is defined as the closed subspace of $L^{2}(\mathbb{R}^{n})$ whose Fourier transform is supported in $\Omega$. Namely, $$\PW(\Omega)=\{f\in L^{2}(\mathbb{R}^{n}):\supp\widehat{f}\subset\Omega\}.$$ For a tempered distribution $\phi\in S'(\mathbb{R}^{n})$ we define the Hankel operator on $\PW(\Omega)$ to be the operator given by $$\widehat{\Ha_{\phi}f}(x)=\int_{\Omega}\widehat{f}(y)\widehat{\phi}(x+y)dy,\text{   for $x\in\Omega$},$$
which is densely defined on smooth functions with compact Fourier support in $\Omega$. Every bounded function $\phi$ generates a bounded Hankel operator on $\PW(\Omega)$, and moreover, the estimate  $\|\Ha_{\phi}\|\leq \|\phi\|_{L^{\infty}}$ holds. Thus a very natural question arises: 
\begin{q}
	Is every bounded Hankel operator generated by a bounded function?
\end{q}
We say that the Nehari theorem is true for $\PW(\Omega)$ if for any $\phi\in S'(\mathbb{R}^{n})$ that generates a bounded Hankel operator $\Ha_{\phi}$ on $\PW(\Omega)$ there exists a bounded function $\psi$, such that $\widehat{\phi}=\widehat{\psi}$ in $2\Omega$, and thus $\Ha_{\phi}=\Ha_{\psi}$. 
\par 
This problem was originally stated for the Hardy space $\mathcal{H}^{2}$ of the two dimensional unit disc. Assume $\{a_{n}\}_{n\in\mathbb{N}_{0}}\subset\mathbb{C}$ and define the Hankel operator $\Ha:\mathcal{H}^{2}\to\mathcal{H}^{2}$ with infinite matrix $(a_{n+m})_{n,m\geq 0}$ with respect to the basis $\{z^{n}:n\in \mathbb{N}_{0}\}$ of $\mathcal{H}^{2}$. Z. Nehari \cite{MR82945} proved that this operator is bounded if and only if there exists $\phi\in L^{\infty}(\mathbb{T})$ such that $\widehat{\phi}(n)=a_{n}$ for all $n\geq 0$. Kronecker \cite{MR0237286} gave a characterization of Hankel operators of finite rank. P. Hartman \cite{MR108684} characterized the compact Hankel operators on $\mathcal{H}^{2}$ as all the operators generated by continuous functions on $\mathbb{T}$. V. Peller \cite{MR1949210} and R. Rochberg \cite{MR674875} characterized the Schatten von-Neumann classes for all $p\geq 1$ of Hankel operators, and V. Peller for all $p>0$ \cite{MR1949210}.
\par
 These results were translated to Paley-Wiener spaces by R. Rochberg \cite{MR878246} on a bounded interval on $\mathbb{R}.$ Hankel operators in several variables of the type studied here were introduced by F. Andersson and M. Carlsson in \cite{MR3355784}, \cite{MR3626673}, who among other things characterized their finite rank structure. L. Peng \cite{MR953994},\cite{MR1001657} showed that the Schatten von-Neumann classes of Hankel operators on the two dimensional unit disc and the $n-$dimensional cube $(-\pi,\pi)^{n}$ are generated by distributions in Besov spaces. For a set $\Omega\subset\mathbb{R}^{n}$, the Besov spaces $B_{p}(\Omega)$, $p>0$ on Paley-Wiener space of $\Omega$ are defined with respect to a decomposition and a partition of unity of $\Omega$. L. Peng showed that for $(-\pi,\pi)^{n}$ it is true that $\Ha_{\phi}\in S_{p}$ if and only if $\phi\in B_{p}(2\Omega)$ for all $p>0$. For the two dimensional disc he proved an analogous result for all $p\in[1,2]$ and for $p\geq 2$ he proved necessity. 
 \par
  A simple polytope in $\mathbb{R}^{n}$ is defined to be a polytope such that every vertex belongs to exactly $n$ edges (an edge is an $n-1$ dimensional intersection of a hyperplane with the convex set). K. Perfekt and M. Carlsson \cite{MR0237286} proved that the Nehari theorem  holds for any convex simple polytope in $\mathbb{R}^{n}$, using as a blackbox that it holds for $\mathbb{R}^{n}_{+}$. It was recently discovered that the validity of the Nehari theorem for $\mathbb{R}_{+}^{n}$ is actually an open question \cite{MR4234971}; neither it is known whether there exists a polytope in $\mathbb{R}^{n}$, $n>1$ such that the Nehari theorem holds. For the unit disc in two dimensions it was disproven by K. Perfekt and O. Brevig \cite{MR4502777}. 
 \par 
 The main goal of this paper is to prove that the Nehari theorem fails for all bounded convex non-polytope domains and we achieve that by using convex analysis to generalize the argument of \cite{MR4502777}. Furthermore, we prove that in two dimensions the Nehari theorem fails for all non-polyhedral sets, i.e. sets that are not a finite intersection of half-spaces (see Section \ref{sec 3}).
 \subsection*{Acknowledgements} The author would like to thank his supervisor, Karl-Mikael Perfekt, for his guidance and the interesting conversations on this topic. He would also like to thank the anonymous referee for valuable comments that helped to improve the paper.
 	\section{Preliminary Lemmas}
 Throughout this section, $\Omega$ will be a non-empty, open and convex subset of $\mathbb{R}^{n}$. Now let us assume that the Nehari theorem holds for $\Omega$. By an open mapping theorem argument, we can find $C>0$ (Lemma \ref{lemma 1}), such that $\inf\{\|\psi\|_{L^{\infty}}:\widehat{\psi}|_{2\Omega}=\widehat{\phi}|_{2\Omega}\}\leq C\|\Ha_{\phi}\|$. 
 \begin{defin}
 	For a Schwartz function $\phi$ we will use the following notation:
 	$$D_{\phi}=\{x\in\Omega:x+y\in \supp\widehat{\phi},\text{ for some $y\in\Omega$}\}=(\supp\widehat{\phi}-\Omega)\cap\Omega.$$
 \end{defin}
 Observe that if $\phi_{1}$ and $\phi_{2}$ have disjoint $D_{\phi_{1}}$ and $D_{\phi_{2}}$, then $$\|\Ha_{\phi_{1}+\phi_{2}}\|=\max(\|\Ha_{\phi_{1}}\|,\|\Ha_{\phi_{2}}\|) \quad (\text{Lemma \ref{lemma 2}}).$$ Using this, we construct a sequence of functions $\phi_{n}$, such that $\psi_{n}=\phi_{1}+...+\phi_{n}$ satisfies
$$\dfrac{\inf\{\|\psi\|_{L^{\infty}}:\widehat{\psi}|_{2\Omega}=\widehat{\psi_{n}}|_{2\Omega}\}}{\|\Ha_{\psi_{n}}\|}\to\infty,$$ which disproves the Nehari theorem. The construction of these functions is based on the existence of infinitely many exposed points in $\overline{\Omega}$ (boundary points with supporting hyperplane that intersects the set only once). 
\begin{example}
	In the case of $\Omega$ being the unit disc $\mathbb{D}=\{z\in\mathbb{C}:|z|<1\}$ and $\supp\widehat{\phi}=\overline{D}(z,r)$ the disc centered at $z\in\mathbb{C}$ with radius $r>0$, a simple computation gives $D_{\phi}=\mathbb{D}\cap D(z,1+r)$. If we take $z$ to approach a point $w$ in the boundary of $2\mathbb{D}$ and $r\to 0$, then $D_{\phi}$ approaches the singleton set $\{\frac{1}{|w|}w\}.$ 
\end{example}
A key idea in \cite{MR4502777} is that the boundary of $2\mathbb{D}$ has infinitely many points with the property that $D_{\phi}$ can approach a singleton.
The following example shows that there are convex sets with only finitely many such points.
  
\begin{example}
	Let us consider now $\Omega=(-1,1)^{n}$ and $\supp\widehat{\phi}=[x_{1},y_{1}]\times...\times [x_{n},y_{n}]$. Then we can calculate $$D_{\phi}=(-1,1)^{n}\cap (x_{1}-1,y_{1}+1)\times...\times(x_{n}-1,y_{n}+1).$$ Let us choose a point in the boundary of $[-2,2]^{n}$ that is not a vertex, say $z=(z_{1},...,z_{n})$. Then it will have a coordinate $-2<z_{i}<2$ and thus if we take the center of the cube $ [x_{1},y_{1}]\times...\times [x_{n},y_{n}]$ to approach $z$ and its diameter to approach to zero, then the $i-th$ coordinate of $D_{\phi}$ will approach the line segment $(-1,1)\cap (z_{i}-1,z_{i}+1),$ which is not a singleton. 
\end{example}
\par The difference between the two examples is that the cube has finitely many extreme points in contrast with the disc that has infinitely many. Polytopes are the only bounded convex sets with finitely many extreme points (see Section \ref{sec 3}). Now we are going to present some technical lemmas that we will use for our theorem's proof. 
	\begin{lm} \label{lemma 1}
	If the Nehari theorem is true for $\Omega$, then there exists $C>0$ such that 
	$$\inf\{\|\psi\|_{L^{\infty}}:\widehat{\psi}|_{2\Omega}=\widehat{\phi}|_{2\Omega}\}\leq C\|\Ha_{\phi}\|$$
	\end{lm}
\begin{proof}
	Let $\phi\sim\psi$ if and only if $\widehat{\psi}|_{2\Omega}=\widehat{\phi}|_{2\Omega}$. Then $\sim$ is an equivalence relation on $L^{\infty}(\mathbb{R}^{n})$.
	 We define a norm on $L^{\infty}(\mathbb{R}^{n})/\sim$ as follows:
	$$\|[\phi]\|:=\inf\{\|\psi\|_{L^{\infty}}:\psi\sim\phi\},$$ which makes $L^{\infty}/\sim$ a Banach space.
	From \cite[Proposition 5.1]{MR4227573} the set of Hankel operators $\mathscr{H}$ is a closed subspace of the bounded operators on $\PW(\Omega)$ and thus, a Banach space. The Nehari theorem says that the operator $T:L^{\infty}(\mathbb{R}^{n})/_{\sim}\to \mathscr{H}$ with $T(\phi)=\Ha_{\phi}$ is bijective and continuous, and thus the inverse operator is bounded.
\end{proof}

\begin{lm} \label{lemma 2}
	Let $\phi=\phi_{1}+\phi_{2}$, such that $D_{\phi_{1}}\cap D_{\phi_{2}}=\emptyset.$ Then
	 $$\|\Ha_{\phi}\|=\max\{\|\Ha_{\phi_{1}}\|,\|\Ha_{\phi_{2}}\|\}.$$
\end{lm}
Note that by induction, this lemma extends to finite sums.
	\begin{proof}
	Let us define $T_{i}\in B(\PW(D_{\phi_{i}}))$ by $T_{i}f=\Ha_{\phi_{i}}(f),$ $i=1,2$ and $T\in B(\PW(D_{\phi_{1}}\cup D_{\phi_{2}}))$ by $Tf=\Ha_{\phi}(f)$. First, let us observe that $\|T_{i}\|=\|\Ha_{\phi_{i}}\|$ and $\|T\|=\|\Ha_{\phi}\|$. This is because:
	\begin{eqnarray} 
	\|\Ha_{\phi_{i}}\|&=&\sup_{f\in \PW(\Omega)}\dfrac{\|\Ha_{\phi_{i}}f\|_{L^{2}}}{\|f\|_{L^{2}}}=\sup_{f\in \PW(\Omega)}\dfrac{\|\Ha_{\phi_{i}}(f\ast\mathscr{F}^{-1}(\chi_{D_{\phi_{i}}}))\|_{L^{2}}}{\|\widehat{f}\|_{L^{2}}}\nonumber \\
	&\leq&\sup_{f\in \PW(\Omega)}\dfrac{\|T_{i}(f\ast\mathscr{F}^{-1}(\chi_{D_{\phi_{i}}}))\|_{L^{2}}}{\|\widehat{f}\chi_{D_{\phi_{i}}}\|_{L^{2}}}=\sup_{f\in \PW(\Omega)}\dfrac{\|T_{i}(f\ast\mathscr{F}^{-1}(\chi_{D_{\phi_{i}}}))\|_{L^{2}}}{\|f\ast\mathscr{F}^{-1}(\chi_{D_{\phi_{i}}})\|_{L^{2}}}\nonumber \\
	&=&\sup_{f\in \PW(D_{\phi_{i}})}\dfrac{\|T_{i}(f)\|_{L^{2}}}{\|f\|_{L^{2}}}=\|T_{i}\|.\nonumber 
	\end{eqnarray}
    The converse inequality follows from the fact that $\PW(D_{\phi_{i}})\subset \PW(\Omega)$. Using the same arguments we prove the same result for $T$. Now since $\PW(D_{\phi_{1}})\oplus \PW(D_{\phi_{2}})=\PW(D_{\phi_{1}}\cup D_{\phi_{1}})$ and 
    for $f=f_{1}+f_{2}$ being the decomposition of $f$ with respect to the direct sum. Then $Tf=T_{1}f_{1}+T_{2}f_{2}$ and thus $T=T_{1}\oplus T_{2}$. This implies that $\|T\|=\max(\|T_{1}\|,\|T_{2}\|),$ as desired.
\end{proof}
	\begin{lm} \label{lemma 3}
		Let $\widehat{f}\in C_{c}(\mathbb{R}^{n})$ be such that $\supp\widehat{f}\subset2\Omega$. Then 
	    $$\dfrac{|\langle\widehat{f},\widehat{\phi}\rangle|}{\|f\|_{L^{1}}}\leq \inf\{\|\psi\|_{L^{\infty}}:\widehat{\psi}|_{2\Omega}=\widehat{\phi}|_{2\Omega}\}.$$
	\end{lm}
\begin{proof}
	$$\dfrac{|\langle\widehat{f},\widehat{\phi}\rangle|}{\|f\|_{L^{1}}}=\dfrac{|\langle\widehat{f},\widehat{\psi}\rangle|}{\|f\|_{L^{1}}}=\dfrac{|\langle f, \psi\rangle|}{\|f\|_{L^{1}}}\leq \dfrac{\|f\|_{L^{1}}\|\psi\|_{L^{\infty}}}{\|f\|_{L^{1}}}=\|\psi\|_{L^{\infty}}.$$
	Taking infimum over all such $\psi$, we get the result.
\end{proof}
The next proposition systematizes the approach of the main idea from \cite{MR4502777}.
	\begin{prop} \label{prop 1}
		If for every $k\in\mathbb{N}$ we can find $y_{i}\in2\Omega$, $i=1,...,k$ and $r>0$ such that 
		$$(\overline{B}(y_{i},r)-\Omega)\cap (\overline{B}(y_{j},r)-\Omega)\cap \Omega=\emptyset,$$ for all $i\neq j$, then the Nehari theorem fails for $\Omega$. 
	\end{prop}
\begin{proof}
	First, note that, if such $r>0$ exists, we can choose $r$ small enough such that also $\overline{B}(y_{i},r)\subset 2\Omega$.
	Let $\widehat{b}$ be a smooth bump function such that $\supp\widehat{b}=\overline{B}(0,1)$, $\widehat{b}=1$ on $\overline{B}(0,\frac{1}{2})$ and $\widehat{b}(x)\in [0,1]$ for all $x\in\mathbb{R}^{n}$. Let $\widehat{b}_{j}^{r}(x)=\widehat{b}(\frac{x-y_{j}}{r})$, $j\in\{1,...,k\}$. For a fixed $k\in\mathbb{N}$, let $r>0$ be such that our assumption holds and let $\phi_{k}=\sum_{j=1}^{k}b_{j}^{r}$.
	We will prove that $$\dfrac{\|\widehat{\phi_{k}}\|_{L^{2}}^{2}}{\|\Ha_{\phi_{k}}\|\|\phi_{k}\|_{L^{1}}}\to \infty$$ and thus by Lemmas \ref{lemma 1} and \ref{lemma 3} the Nehari theorem fails for $\Omega$. Since $\overline{B}(y_{j},r)\subset 2\Omega$, we get that $\frac{1}{2}\overline{B}(y_{j},r)\subset \Omega$ and $\frac{1}{2}\overline{B}(y_{j},r)\subset \overline{B}(y_{j},r)-\frac{1}{2}\overline{B}(y_{j},r)\subset \overline{B}(y_{j},r)-\Omega$ and thus $$\supp \widehat{b}_{j}^{r}=\overline{B}(y_{j},r)\subset 2(\Omega\cap(\overline{B}(y_{j},r)-\Omega)).$$ Therefore by our assumption $b_{j}^{r}$ are orthogonal, hence we estimate 
	$$\|\widehat{\phi_{k}}\|_{L^{2}}^{2} =\sum_{j=1}^{k}\|\widehat{b}_{j}^{r}\|_{L^{2}}^{2}=r^{n}k\|\widehat{b}\|_{L^{2}}^{2}.$$
	Since now $D_{b_{j}^{r}}=(\overline{B}(y_{j},r)-\Omega)\cap\Omega$, for the norms of the Hankel operators, using Lemma \ref{lemma 2} we get 
	$$\|\Ha_{\phi_{k}}\|=\max_{1\leq j\leq k}\|\Ha_{b_{j}^{r}}\|\leq \max_{1\leq j\leq k}\|b_{j}^{r}\|_{L^{\infty}}\leq \max_{1\leq j\leq k}\|\widehat{b}_{j}^{r}\|_{L^{1}}= r^{n}\|\widehat{b}\|_{L^{1}}.$$
	Now, we need to bound $\|\phi_{k}\|_{L^{1}}$ from above. For $R>0$, we decompose this norm into:
	$$\int_{\mathbb{R}^{n}}|\phi_{k}|=\int_{|x|\leq \frac{R}{r}}|\phi_{k}|+\int_{|x|>\frac{R}{r}}|\phi_{k}|:=I_{1}+I_{2}.$$
	For $I_{1}$ we use the Cauchy-Schwarz inequality to get
	$$I_{1}\leq \sqrt{m_{n}(\frac{R}{r}B_{n})}\|\phi_{k}\|_{L^{2}}=c_{n}\left(\dfrac{R}{r}\right)^{\frac{n}{2}}\sqrt{\sum_{j=1}^{k}\|b_{j}^{r}\|^{2}_{L^{2}}}= c_{n}\left(\dfrac{R}{r}\right)^{\frac{n}{2}}\sqrt{kr^{n}\|b\|_{L^{2}}^{2}}= c_{n}R^{\frac{n}{2}}\sqrt{k\|b\|_{L^{2}}^{2}},$$
	where $m_{n}$ is the Lebesgue measure on $\mathbb{R}^{n}$, $B_{n}$ is the unit ball in $\mathbb{R}^{n}$ and $c_{n}$ is the square root of the volume of $B_{n}$.

	For $I_{2}$ we observe that the definition $\widehat{b}_{j}^{r}(x)=\widehat{b}(\frac{x-y_{j}}{r})$ implies that $b_{j}^{r}(x)=r^{n}e^{\frac{2\pi i}{r}xy_{j}}b(rx)$ and thus we can calculate
	\begin{eqnarray}
	I_{2}&=&\int_{|x|>\frac{R}{r}}|\phi_{k}|\leq\sum_{j=1}^{k}\int_{|x|>\frac{R}{r}}|b_{j}^{r}|=\sum_{j=1}^{k}\int_{|x|>\frac{R}{r}}|r^{n}b(rx)|=\sum_{i=1}^{k}\int_{|x|>R}|b|=k\int_{|x|>R}\dfrac{|x||b|}{|x|} \nonumber \\
	&\leq& \dfrac{k}{R}\||x|b\|_{L^{1}},\nonumber
	\end{eqnarray}
where $\||x|b\|_{L^{1}}$ is finite since $b$ is a Schwartz function.

	Finally, if we assume the Nehari theorem on $\PW(\Omega)$, by using Lemmas \ref{lemma 1} and \ref{lemma 3} we get that for every $k\in\mathbb{N}$ and $R>0$: 
	\begin{eqnarray}
	C&\geq& \dfrac{\inf\{\|\psi\|_{L^{\infty}}:\widehat{\psi}|_{2\Omega}=\widehat{\phi_{k}}|_{2\Omega}\}}{\|\Ha_{\phi_{k}\|}}\geq \dfrac{\|\widehat{\phi_{k}}\|_{L^{2}}^{2}}{\|\phi_{k}\|_{L^{1}}\|\Ha_{\phi_{k}}\|}\geq \dfrac{r^{n}k\|\widehat{b}\|_{L^{2}}^{2}}{r^{n}\|\widehat{b}\|_{L^{1}}(c_{n}R^{\frac{n}{2}}\sqrt{k\|b\|_{L^{2}}^{2}}+\dfrac{k}{R}\||x|b\|_{L^{1}})}\nonumber \\ &=&\dfrac{k\|\widehat{b}\|_{L^{2}}^{2}}{\|\widehat{b}\|_{L^{1}}(c_{n}R^{\frac{n}{2}}\sqrt{k\|b\|_{L^{2}}^{2}}+\dfrac{k}{R}\||x|b\|_{L^{1}})}.\nonumber
\end{eqnarray}
 Now, taking $k\to\infty$ we find that
 $$C\geq \dfrac{R\|\widehat{b}\|_{L^{2}}^{2}}{\|\widehat{b}\|_{L^{1}}\||x|b\|_{L^{1}}}$$
 for every $R>0$, thus the Nehari theorem fails.
\end{proof}

\section{The main theorem and its proof}\label{sec 3}
Let $K\subset \mathbb{R}^{n}$ be closed with nonempty interior. For $x\in\partial K$, a supporting hyperplane of $K$ through $x$ is a hyperplane $H=\{y\in\mathbb{R}^{n}:\langle y,a\rangle= t\}$, $a\in\mathbb{R}^{n},$ $t\in\mathbb{R}$, such that 
$$K\subset H^{-}=\{y\in\mathbb{R}^{n}:\langle y,a\rangle\leq t\}\quad \text{and}\quad x\in H.$$ $K$ is convex if and only if every boundary point admits a supporting hyperplane  \cite[Theorem 1.16]{MR4180684}. There might exist more than one supporting hyperplane through some $x\in\partial K$, for example let us consider consider $K=[-1,1]^{2}$ in $\mathbb{R}^{2}$. Then the lines $x=1$ and $y=1$ are both supporting hyperplanes of $K$ through $(1,1)$. If $K\cap H^{-}=\{x\}$ for some supporting hyperplane, then we call $x$ an exposed point of $K$. For two points $x_{1},x_{2}\in \mathbb{R}^{n}$ we use the notation $$[x_{1},x_{2}]=\{tx_{1}+(1-t)x_{2}:t\in [0,1]\}\quad\text{and}\quad (x_{1},x_{2})=\{tx_{1}+(1-t)x_{2}:t\in (0,1)\}$$ We say that a point $x\in\partial K$ is an extreme point of $K$, if for every $y,z\in K$, $x\in [y,z]$ implies that $x=y$ or $x=z$. It is easy to see that an exposed point is an extreme point. The converse is not true. For example we can consider the set $$[-2,0]\times [-1,1]\cup\{(x,y)\in\mathbb{R}^{2}:x^{2}+y^{2}\leq 1\text{ and } x\geq 0 \}.$$ The point $(0,1)$ is an extreme point but the only supporting hyperplane through $(0,1)$ is $y=1$ which intersects our set at $[-2,0]\times\{1\}$. However, there is a relation between extreme and exposed points, namely a theorem of Straszewicz \cite{MR0274683} that states the following:
\begin{lm}(Straszewicz's theorem) \label{Straszewicz's theorem}
	Let $K$ be a convex closed set. Then the set of exposed points are dense in the set of extreme points.
\end{lm}
 For a non-empty convex set $K\subset \mathbb{R}^{n}$ we will denote by $\ext K$ and $\exp K$ the extreme and the exposed points of $\overline{K}$ respectively. 
	\begin{theo} \label{theo 1}
		Let $\Omega\subset\mathbb{R}^{n}$ be non-empty, open and convex, and suppose that $\exp\Omega$ is an infinite set. Then the Nehari theorem fails for $\Omega$.
		
	\end{theo}
		\begin{proof} We want to prove the following claim:
			\begin{Claim}
				For every $y\in\exp\Omega$ and every $\rho>0$ there exists $z\in\Omega$ and $s>0$ such that 
				$$(\overline{B}(2z,s)-\Omega)\cap\Omega\subset B(y,\rho).$$  
			\end{Claim}
		First let us suppose the claim is true. Then, for $k\in\mathbb{N}$ we can choose distinct $y_{1},...,y_{k}\in\exp\Omega$ and $\rho=\frac{1}{2}\min_{1\leq i\neq j\leq k}|y_{i}-y_{j}|$. Then by our claim we can find $z_{i}\in \Omega$ and $s>0$ such that $$(\overline{B}(2z_{i},s)-\Omega)\cap (\overline{B}(2z_{j},s)-\Omega)\cap \Omega\subset B(y_{i},\rho)\cap B(y_{j},\rho)=\emptyset,\quad i\neq j.$$ Finally, Proposition \ref{prop 1} yields the result.
		 \\\textbf{Proof of Claim:}
			Let us fix $y\in\exp\Omega$ and assume that our claim is false. Let us also fix $x_{k}\in\Omega$, $x_{k}\to y$. Then there exists a $\rho>0$ such that for all $k\in\mathbb{N}$,
			$$(\overline{B}(2x_{k},\dfrac{1}{k})-\Omega)\cap\Omega\nsubseteq B(y,\rho).$$ Thus we can find a sequence $z_{k}\in\Omega$ such that
			$|z_{k}-y|\geq\rho$ and $z_{k}=2x_{k}+\dfrac{1}{k}u_{k}-\omega_{k}$, where $\omega_{k}\in \Omega$ and $|u_{k}|\leq 1.$	
			Let $H=\{x\in\mathbb{R}^{n}:\langle x,a \rangle= t\}$ be a supporting hyperplane at $y$ that intersects $\overline{\Omega}$ only at $y$. Then
			\begin{eqnarray} 
			t&>&\langle z_{k},a\rangle =\langle 2x_{k}+\dfrac{1}{k}u_{k}-\omega_{k},a\rangle = 2\langle x_{k},a\rangle+\dfrac{1}{k}\langle u_{k},a\rangle-\langle \omega_{k},a\rangle\geq 2\langle x_{k},a\rangle+\dfrac{1}{k}\langle u_{k},a\rangle-t.\nonumber 
			\end{eqnarray}
			Taking $k\to \infty$ we get that $\lim_{k}\langle z_{k},a\rangle =t$. Since $z_{k}\notin B(y,\rho)$, we can find $a_{k}\in [z_{k},y]\subset \overline{\Omega}$ such that $a_{k}\in \partial B(y,\rho)$. 
			Now let $a_{k_{i}}$ be a subsequence that converges to some $a_{0}\in \partial B(y,\rho)$. Since $\overline{\Omega}$ is closed, $a_{0}\in \overline{\Omega}$. It must also be the case that $\langle a_{0},a\rangle=t$, since:	
		    There exists $t_{k_{i}}\in[0,1]$ such that $a_{k_{i}}=t_{k_{i}}z_{k_{i}}+(1-t_{k_{i}})y$, and if we assume $k_{j}$ to be a subsequence of $k_{i}$ such that $t_{k_{j}}$ converges, we get that 
			$$\langle a_{0},a\rangle =\lim_{j} \langle a_{k_{j}},a\rangle= \lim_{j}(t_{k_{j}}\langle z_{k_{j}},a \rangle +(1-t_{k_{j}}) \langle y,a \rangle )=t.$$
		    Therefore $a_{0}\in H\cap \overline{\Omega}$ which implies that $a_{0}=y$, contradicting the fact that $a_{0}\in\partial B(y,\rho)$.
		\end{proof}
	A polytope is defined to be the convex hull of finitely many points, i.e. the smallest convex set that contains these points, and a polyhedron to be the intersection of finitely many half-spaces. These two definition coinside for bounded convex sets \cite[Theorems 1.20 and 1.22]{MR4180684}. For a polytope $K$ generated by a finite set $E$ it is true that $\exp K=\ext K\subset E$ \cite[Theorem 1.18]{MR4180684}. A closed bounded convex set in $\mathbb{R}^{n}$ is the convex hull of its extreme points \cite[Theorem 1.21]{MR4180684}, therefore polytopes are the only bounded convex sets with finitely many extreme points, and consequently finitely many exposed points (Lemma  \ref{Straszewicz's theorem}). Thus we get the following corollary:
\begin{cor}\label{cor 1}
	The Nehari theorem fails for any convex set with an infinite set of extreme points. In particular the Nehari theorem fails for any bounded convex set that is not a polytope.
\end{cor}
\begin{proof}
	Let $\Omega$ be such a set. Since $\overline{\Omega}$ is closed and has an infinite set of extreme points, the set of exposed points must also be infinite by Lemma \ref{Straszewicz's theorem}.
	\end{proof}

In the case of unbounded convex sets, every closed convex set $K$ that does not contain lines is the convex hull of its extreme points and its extreme half-lines \cite[Theorem 2]{MR2463172}(i.e. a half-line $l\subset \partial K$ with the property that if two points $x,y$ in our convex set satisfy $(x,y)\cap l\neq\emptyset$, then both $x$ and $y$ must belong to $l$). Using this we prove the following corollary:
\begin{cor}\label{cor 2}
	The Nehari theorem fails for every convex subset of $\mathbb{R}^{2}$ that is not a polyhedron.
\end{cor} 
In order to prove this corollary, we first need to prove the following lemma:
\begin{lm} \label{lm 5}
	Let $K\subset \mathbb{R}^{2}$ be convex and closed. Then $K$ has finitely many extreme half-lines. 
\end{lm}
\begin{proof}
 Suppose that $K$ contains 9 or more extreme half-lines. Then by the Pigeonhole principle there exist three of them with directions in the same quadrant, say the first one. Thus $\ell_{j}=\{y=a_{j}x+b_{j}:x\geq x_{j}\}$, for some $a_{j}\geq 0$ and $b_{j},x_{j}\in\mathbb{R}$. Let $\ell=\{y=-x+b\}$, $b\in\mathbb{R}$, be a line such that $(x_{j},a_{j}x_{j}+b_{j})\in\{y\leq-x+b\}$ (for example we can choose $b=\max_{j=1,2,3}(x_{j}+ax_{j}+b_{j})$). Let us set $f_{j}(x)=x+a_{j}x+b_{j}$. Then $f_{j}(x_{j})\leq b$ and $\lim_{x\to +\infty}f_{j}(x)=+\infty$ and thus by the intermediate value theorem, for every $j$ we can find $z_{j}\geq x_{j}$ such that $f_{j}(z_{j})=b$ which implies that $(z_{j},a_{j}z_{j}+b_{j})\in\ell_{j}\cap\ell$. Since these three points are on the same line, we can assume without loss of generality that $(z_{1},a_{1}z_{1}+b_{1})$ is strictly in the line segment that connects $(z_{2},a_{2}z_{2}+b_{2})$ and $(z_{3},a_{3}z_{3}+b_{3})$ which contradicts the fact that $\ell_{1}$ is extremal.
\end{proof}
For the proof of the corollary, we need to introduce the following notation: For a set $X\subset\mathbb{R}^{n}$ we define the positive cone generated by $X$ to be the set $$\cone X=\{\sum_{i=1}^{m}\mu_{i}x_{i}:\mu_{i}\geq 0,x_{i}\in X,m\in\mathbb{N}\}.$$ By \cite[Theorem 1.2]{MR1311028}, we know that a convex set in $\mathbb{R}^{n}$ is a polyhedron if and only if it is the sum of a polytope and a positive cone generated by a finite set. Thus we have the following proof:

\textbf{Proof of Corollary 8:}
 It suffices to prove that if a closed set $K$ has finitely many extreme points then it is a polyhedron. The case where $K$ is bounded has been covered in Corollary \ref{cor 1}. Thus for the rest of the proof, $K$ will be unbounded with finitely many extreme points. First, let us assume that our set $K$ contains a line $\ell$. Let, without loss of generality $\ell=\mathbb{R}\times\{0\}$. If $(x,y)\in K \smallsetminus \ell$, by convexity of $K$ we get that $\mathbb{R}\times[0,y]\subset K$. Therefore, by letting $\alpha=\sup_{(x,y)\in K} y$ and $\beta=\inf_{(x,y)\in K}y$ we get $K=\mathbb{R}\times [\beta,\alpha]$ and thus a polyhedron. Therefore we have to examine the case where $K$ does not contain lines, hence $K$ is the convex hull of its extreme points and its extreme half-lines. By Lemma \ref{lm 5}, $K$ can be written as the convex hull of finitely many points and finitely many half-lines, namely $$K=\conv \{x_{1},...,x_{m},\ell_{1},...,\ell_{k}\},$$
where $m\geq k$ and $x_{i}$ is the origin of $\ell_{i},$ $i=1,...,k.$ Let $u_{i}$ be the direction of $\ell_{i}$. We will prove that
$$K=\conv\{x_{1},...,x_{m},\ell_{1},...,\ell_{k}\}=\conv\{x_{1},...,x_{m}\}+\cone\{u_{1},...,u_{k}\}$$ and thus by \cite[Theorem 1.2]{MR1311028} $K$ is a polyhedron. For the inclusion $``\subset"$, we observe that $x_{i}\in \conv\{x_{1},...,x_{m}\}+\{0\}\subset \conv\{x_{1},...,x_{m}\}+\cone\{u_{1},...,u_{k}\}$ and $\ell_{i}=\{x_{i}+\mu_{i}u_{i}:\mu_{i}\geq 0 \}\subset  \conv\{x_{1},...,x_{m}\}+\cone\{u_{1},...,u_{k}\}$. The inclusion follows from the fact that the positive cone of a set is convex and that the sum of convex sets is also convex. For the other inclusion, let $x\in \conv\{x_{1},...,x_{m}\}+\cone\{u_{1},...,u_{k}\}$. Then there exist $p_{i},\mu_{i}\geq 0$, such that $\sum_{i=1}^{m}p_{i}=1$ and $x=\sum_{i=1}^{m}p_{i}x_{i}+\sum_{i=1}^{k}\mu_{i}u_{i}$. Finally we write $x$ as 
	$$x= \sum_{i=k+1}^{m}p_{i}x_{i}+\sum_{i=1,p_{i}\neq 0}^{k}p_{i}(x_{i}+\dfrac{\mu_{i}}{p_{i}}u_{i})\in \conv\{x_{1},...,x_{m},\ell_{1},...,\ell_{k}\},$$
	as desired.
\qed
\vspace{1mm}
\par  The argument in Corollary \ref{cor 2} cannot be extended to $\mathbb{R}^{n}$, $n\geq 3$. Corollary \ref{cor 2} uses the fact that in $\mathbb{R}^{2}$ there is no convex set with finitely many extreme points that is not a polyhedron. One counterexample in $\mathbb{R}^{3}$ would be the cone $\{(x,y,z)\in\mathbb{R}^{3}:x^{2}+y^{2}=z^{2},z\geq0\}$. This set has only one extreme point (i.e. the origin) and infinitely many extreme half-lines $\{\lambda(\cos\theta,\sin\theta,1):\lambda\geq 0\}$ but it is not a polyhedron. Therefore one could ask the following question:
Is there a convex set in $\mathbb{R}^{n}$, $n\geq 3$ that is not a polyhedron, such that the Nehari theorem holds?

	\bibliographystyle{plain}
	\bibliography{Bibliography}

	\end{document}